\def\benm{\begin{enumerate}}
\def\eenm{\end{enumerate}}

\def\bal{\begin{align}}
\def\eal{\end{align}}

\def\ii{\mathrm{i}}

\documentclass[11pt]{amsart}
\usepackage{amsmath, amsthm, amscd, amsfonts,amssymb, graphicx}
\usepackage{color}

\newtheorem{theorem}{Theorem}[section]

\theoremstyle{definition}

\newtheorem{example}[theorem]{Example}

\newtheorem{proposition}[theorem]{Proposition}
\newtheorem{corollary}[theorem]{Corollary}

\theoremstyle{remark}
\newtheorem{remark}[theorem]{Remark}

\numberwithin{equation}{section}

\newcommand{\bfomega}{\mbox{\boldmath $\omega$ \unboldmath} \hskip -0.05 true in}
\newcommand{\bom}{\bfomega}

\newcommand{\dd}{\mathrm{d}}

\begin{document}

\title[Banach Modules of Group Algebras on Covariant Functions]{Banach Convolution Modules of Group Algebras on Covariant Functions of Characters of Normal Subgroups}

\author[A. Ghaani Farashahi]{Arash Ghaani Farashahi}
\address{Department of Pure Mathematics, School of Mathematics, University of Leeds, Leeds LS2 9JT, United Kingdom}
\email{a.ghaanifarashahi@leeds.ac.uk}
\email{ghaanifarashahi@outlook.com}

\curraddr{}


%   Information for second author


%   General info

%\date{\bf\today}


%\thanks{$^*$Corresponding author}
\thanks{E-mail addresses: a.ghaanifarashahi@leeds.ac.uk (Arash Ghaani Farashahi) }

\begin{abstract}
This paper investigates structure of Banach convolution modules induced by group algebras on covariant functions of characters of closed normal subgroups. Let $G$ be a locally compact group with the group algebra $L^1(G)$ and $N$ be a closed normal subgroup of $G$. Suppose that $\xi:N\to\mathbb{T}$ is a continuous character, $1\le p<\infty$ and $L_\xi^p(G,N)$ is the $L^p$-space of all covariant functions of $\xi$ on $G$. It is shown that $L^p_\xi(G,N)$ is a Banach $L^1(G)$-module. We then study convolution module actions of group algebras on covariant functions of characters for the case of canonical normal subgroups in semi-direct product groups. 
\end{abstract}

\maketitle

%\tableofcontents
\section{\bf{Introduction}}
The covariant functions of characters (one-dimensional continuous irreducible unitary representations) of closed subgroups have applications in different mathematical areas such as number theory (automorphic forms), induced representations, homogeneous spaces, complex (hypercomplex) analysis, and coherent states, see \cite{RB, FollP, kan.taylor, kisil.BJMA, kisil.book, kisil1, MackII, MackI, pere}.

In general terms, classical harmonic analysis techniques cannot be employed as a unified theory for covariant functions of a given closed subgroup. In the case of a compact subgroup, harmonic analysis methods for covariant functions of characters discussed in \cite{AGHF.cov.com}. Some operator theoretic aspects related to Banach covariant function spaces of characters of normal subgroups investigated in \cite{AGHF.ch.normal.main}. The following paper studies structure of Banach convolution modules induced by the group algebras on Banach covariant function spaces of characters of normal subgroups. The convolution module action of group algebras on covariant functions extends well-known structures of abstract harmonic analysis in different directions. The module structure coincides with the convolution of integrable functions on the group when the normal subgroup is compact. Further, it also gives canonical extension of classical methods for abstract harmonic analysis of Banach convolution modules on quotient groups when the character is the trivial character of the normal subgroup, see \cite{Bour, Brac, der1983, Fei1, Fei2, kan.lau}. 

This article contains 5 sections and organized as follows. Section  2  is  devoted  to  fix  notations  and  provides  a  summary
of classical harmonic analysis on locally compact compact groups and quotient groups of locally compact groups. 
Suppose that $G$ is a locally compact group with the group algebra $L^1(G)$ and $N$ is a closed normal subgroup of $G$. Let $\xi:N\to\mathbb{T}$ be a fixed character of $N$, and $1\le p<\infty$. Assume that $L^p_\xi(G,N)$ is the $L^p$-space of covariant functions of $\xi$ in $G$. In section 3 we study harmonic analysis on covariant functions of $\xi$. We then consider convolution module action of continuous compactly supported functions on continuous covariant functions of $\xi$. Next we extend the convolution module action to the Banach algebra $L^1(G)$ and the Banach spaces $L^p_\xi(G,N)$. We then discuss abstract harmonic analysis aspects of this extension. It is shown that the extended convolution module action makes $L^p_\xi(G,N)$ into a Banach $L^1(G)$-module. 
Section 5 investigates convolution module action of semi-direct product group algebras on covariant functions of characters of normal subgroups with some examples including abstract Weyl-Heisenberg groups and Heisenberg groups.  

\section{\bf{Preliminaries and Notations}}

Let $X$ be a locally compact Hausdorff space. Then $\mathcal{C}_c(X)$ denotes the space of all continuous complex valued functions on $X$ with compact support. 
Suppose that $\lambda$ is a positive Radon measure on $X$ and $1\le p<\infty$. The Banach space of equivalence classes of $\lambda$-measurable complex valued functions $f:X\to\mathbb{C}$ such that
$$\|f\|_{L^p(X,\lambda)}:=\left(\int_X|f(x)|^p\dd\lambda(x)\right)^{1/p}<\infty,$$
is denoted by $L^p(X,\lambda)$ which contains $\mathcal{C}_c(X)$ as a $\|\cdot\|_{L^p(X,\lambda)}$-dense subspace.

Let $G$ be a locally compact group and $\lambda_{G}$ be a left Haar measure on $G$. For a function $f:G\to\mathbb{C}$ and $x\in G$, the functions $L_xf,R_xf:G\to\mathbb{C}$ are given by $L_xf(y):=f(x^{-1}y)$ and $R_xf(y):=f(yx)$ for $y\in G$.
For $1\le p<\infty$, $L^p(G)$ stands for the Banach function space $L^p(G,\lambda_G)$. The convolution of $f,g\in L^1(G)$ is given by 
\begin{equation}\label{f.ast.g}
f\ast_G g(x):=\int_Gf(y)g(y^{-1}x)\dd\lambda_G(y)\quad (x\in G).
\end{equation}
Then $f\ast_G g\in L^p(G)$ with 
$\|f\ast_G g\|_{L^p(G)}\le \|f\|_{L^1(G)}\|g\|_{L^p(G)}$, if $f\in L^1(G)$ and $g\in L^p(G)$.
It is known in abstract harmonic analysis that the Banach function space $L^1(G)$ is a Banach algebra with respect to the bilinear product $\ast_G:L^1(G)\times L^1(G)\to L^1(G)$ given by $(f,g)\mapsto f\ast_G g$. For $p>1$, the Banach function space $L^p(G)$ is a Banach left $L^1(G)$-module equipped with  
the left module action $\ast_G:L^1(G)\times L^p(G)\to L^p(G)$ 
given by $(f,g)\mapsto f\ast_G g$, see \cite{der00, HR1, 50} and the classical list of references therein. 

Let $N$ be a closed normal subgroup of $G$ and $G/N$ be the factor (quotient) group of $N$ in $G$. Suppose that $\lambda_N$ is a left Haar measure on $N$. The function space $\mathcal{C}_c(G/N)$ consists of all functions $T_N(f)$, where 
$f\in\mathcal{C}_c(G)$ and
\begin{equation*}
T_N(f)(xN)=\int_Nf(xs)\dd\lambda_{N}(s)\quad (xN\in G/N).
\end{equation*}
The quotient group $G/N$ has a left Haar measure $\lambda_{G/N}$, normalized with respect to left Haar measures $\lambda_N$ on $N$ and $\lambda_K$ on $K$, according to the following Weil's formula  
\begin{equation}\label{6}
\int_{G/N}T_N(f)(xN)\dd\lambda_{G/N}(xN)=\int_Gf(x)\dd\lambda_G(x),
\end{equation}
for every $f\in L^1(G)$, see \cite{FollH}. 
Then the convolution module action of $\Psi\in L^1(G/N)$ and $\Phi\in L^p(G/N)$ has the form 
\begin{equation}
\Psi\ast_{G/N}\Phi(xN)=\int_{G/N}\Psi(yN)\Phi(y^{-1}xN)\dd\lambda_{G/N}(yN),
\end{equation}
for $xN\in G/N$ and the linear operator $T_N:L^1(G)\to L^1(G/N)$
is a norm-decreasing homomorphism of Banach algebras, see Theorem 3.5.4 of \cite{50}.   

A character $\xi$ of $N$, is a continuous group homomorphism $\xi:N\to\mathbb{T}$, where $\mathbb{T}:=\{z\in\mathbb{C}:|z|=1\}$ is the circle group. In the context of representation theory for groups, every character of $N$ is a 1-dimensional irreducible continuous unitary representation of $N$. Then we denote the set of all characters of $N$ by $\chi(N)$. A function $\psi:G\to\mathbb{C}$ satisfies covariant property for the character $\xi\in\chi(N)$, if 
\begin{equation}\label{cov.prop.main}
\psi(xs)=\xi(s)\psi(x),
\end{equation}
for $x\in G$ and $s\in N$. In this case, $\psi$ is called a covariant function of $\xi$. 

\section{\bf Abstract Harmonic Analysis of Covariant Functions}
Throughout this section, we study some of the theoretical aspects of covariant functions of characters of closed normal subgroups. The covariant functions arise in abstract harmonic analysis in the construction of induced representations, see \cite{FollH, kan.taylor}. We here exploit one of the classical approaches and tools in this direction. 

Suppose that $G$ is a locally compact group and $N$ is a closed normal subgroup of $G$. Let $\lambda_{G}$ be a left Haar measure on $G$ and $\lambda_N$ be a left Haar measure on $N$.
Assume that $\lambda_{G/N}$ is the left Haar measure on the quotient group $G/N$ normalized with respect to Weil's formula (\ref{6}). For $\xi\in\chi(N)$ and $f\in\mathcal{C}_c(G)$, define the function $T_\xi(f):G\to\mathbb{C}$ by  
\[
T_\xi(f)(x):=\int_Nf(xs)\overline{\xi(s)}\dd\lambda_N(s),\ \ \ {\rm for}\ x\in G. 
\]
Suppose that $\mathcal{C}_\xi(G,N)$ is the linear subspace of $\mathcal{C}(G)$ given by 
\[
\mathcal{C}_\xi(G,N):=\{\psi\in\mathcal{C}_c(G|N):\psi(xk)=\xi(k)\psi(x),\ {\rm for\ all}\ x\in G,\ k\in N\},
\]
where 
\[
\mathcal{C}_c(G|N):=\{\psi\in\mathcal{C}(G):\mathrm{q}({\rm supp}(\psi))\ {\rm is\ compact\ in}\ G/N\},
\]
and $\mathrm{q}:G\to G/N$ is the canonical map given by $\mathrm{q}(x):=xN$ for every $x\in G$. 
Then the linear operator $T_\xi$ maps $\mathcal{C}_c(G)$ onto $\mathcal{C}_\xi(G,N)$, see Proposition 6.1 of \cite{FollH}.

\begin{remark}\label{J1}
If $\xi=1$ is the trivial character of $N$ then $T_1(f)=T_N(f)$.
Then $\mathcal{C}_1(G,N)$ consists of functions on $G$ which are constant on cosets of $N$. Therefore, the function space $\mathcal{C}_1(G,N)$ can be identified with $\mathcal{C}_c(G/N)$ via the identification $\psi\mapsto\widetilde{\psi}$, where $\widetilde{\psi}:G/N\to\mathbb{C}$ is $\widetilde{\psi}(xN):=\psi(x)$ for $x\in G$. 
\end{remark}
For every $1\le p<\infty$ and $\psi\in \mathcal{C}_\xi(G,N)$, define the norm 
\begin{equation}\label{(p)}
\|\psi\|_{(p)}^p:=\left\||\psi|\right\|_{L^p(G/N,\lambda_{G/N})}^p=\int_{G/N}|\psi(y)|^p\dd\lambda_{G/N}(yN).
\end{equation}
For every $\psi\in \mathcal{C}_\xi(G,N)$, the function $y\mapsto |\psi(y)|^p$ reduces to constant on $N$ and so it depends only on the coset $yN$. 
Hence, $yN\mapsto |\psi(y)|^p$ defines a function in $\mathcal{C}_c(G/N)$
which can be integrated with respect to $\lambda_{G/N}$.

\begin{remark}\label{J1.p}
Let $\xi=1$ be the trivial character of $N$. Suppose that  
$\psi\in \mathcal{C}_1(G,N)$ is identified with $\widetilde{\psi}\in\mathcal{C}_c(G/N)$, as in Remark \ref{J1}. Then $\|\psi\|_{(p)}=\|\widetilde{\psi}\|_{L^p(G/N,\lambda_{G/N})}$. This implies that $\psi\mapsto\widetilde{\psi}$ is an isometric identification of $(\mathcal{C}_1(G,N),\|\cdot\|_{(p)})$ by $(\mathcal{C}_c(G/N),\|\cdot\|_{L^p(G/N)})$.
\end{remark}

\begin{proposition}\label{NC.norm.p}
{\it Let $G$ be a locally compact group and $N$ be a compact normal subgroup of $G$. Suppose $\xi\in\chi(N)$, $\psi\in \mathcal{C}_\xi(G,N)$, and $1\le p<\infty$. Then
\begin{equation}\label{Ncompact.Lp.gen}
\|\psi\|_{(p)}=\lambda_N(N)^{1/p}\|\psi\|_{L^p(G)}.
\end{equation} 
}\end{proposition}
\begin{proof}
Invoking Proposition 3.1(1) of \cite{AGHF.cov.com}, we obtain $\mathcal{C}_\xi(G,N)\subseteq\mathcal{C}_c(G)$.  
Let $\psi\in \mathcal{C}_\xi(G,N)$ be given. Using compactness of $N$ in $G$, each Haar measure of $N$ is finite. Then by applying Weil's formula (\ref{6}), we get  
\begin{align*}
\|\psi\|_{L^p(G)}^p
&=\int_{G/N}T_N(|\psi|^p)(xN)\dd\lambda_{G/N}(xN)
\\&=\int_{G/N}\int_N|\psi(xs)|^p\dd\lambda_N(s)\dd\lambda_{G/N}(xN)
\\&=\int_{G/N}\int_N|\psi(x)|^p\dd\lambda_N(s)\dd\lambda_{G/N}(xN)
\\&=\int_{G/N}\left(\int_N\dd\lambda_N(s)\right)|\psi(x)|^p\dd\lambda_{G/N}(xN)
%=\lambda_N(N)\int_{G/N}|\psi(x)|^p\dd\lambda_{G/N}(xN)
=\lambda_N(N)\|\psi\|_{(p)}^p.
\end{align*}
\end{proof}

\begin{remark}
Let $N$ be a compact and normal subgroup of $G$. Suppose that $\lambda_N$ is the probability Haar measure on $N$ and the left Haar measure $\lambda_{G/N}$ on $G/N$ is normalized with respect to Weil's formula (\ref{6}). Then Proposition \ref{NC.norm.p} gives $\|\psi\|_{(p)}=\|\psi\|_{L^p(G)}$, for $1\le p<\infty$ and $\psi\in \mathcal{C}_\xi(G,N)$. In this case, harmonic analysis on covariant functions studied in \cite{AGHF.cov.com}. 
\end{remark}

\section{\bf Convolution Modules of Group Algebras on Covariant Functions}
In this section, we consider the abstract structure of convolution modules on covariant function spaces of characters of normal subgroups induced by group algebras. We then study some of the basic properties of these Banach convolution modules. It is still assumed that $G$ is a locally compact group and $N$ is a closed normal subgroup of $G$. Suppose that $\lambda_{G}$ is a left Haar measure on $G$ and $\lambda_N$ is a left Haar measure on $N$. Let $\lambda_{G/N}$ be the left Haar measure on the quotient group $G/N$ normalized with respect to Weil's formula (\ref{6}) and $\|\cdot\|_{(p)}$ be given by (\ref{(p)}), for every $1\le p<\infty$.
 
For $f\in \mathcal{C}_c(G)$ and $\psi\in \mathcal{C}_\xi(G,N)$, the convolution $f\ast_G\psi:G\to\mathbb{C}$ given by 
\begin{equation}\label{f.psi}
f\ast_G\psi(x):=\int_Gf(y)\psi(y^{-1}x)\dd\lambda_G(x),
\end{equation}
is well-defined, for every $x\in G$. 

If $x\in G$ then the function $g_x:G\to\mathbb{C}$ given by $g_x(y):=f(y)\psi(y^{-1}x)$ is continuous on $G$ with compact support which can be integrated on $G$. Therefore, the right hand side of (\ref{f.psi}) defines a well-defined function on $G$, which we may denote it by just $f\ast \psi$.

\begin{theorem}\label{mult.Cc}
Let $G$ be a locally compact group, $N$ be a closed subgroup of $G$, and $\xi\in\chi(N)$. Suppose $f,g\in\mathcal{C}_c(G)$ are given. Then 
\[
T_\xi(f\ast_G g)=f\ast_G T_\xi(g).
\] 
\end{theorem}
\begin{proof}
Suppose $f,g\in\mathcal{C}_c(G)$ and $x\in G$. We then have  
\begin{align*}
T_\xi(f\ast_G g)(x)
&=\int_N f\ast_G g(xs)\overline{\xi(s)}\dd\lambda_N(s)
\\&=\int_N \left(\int _Gf(y)g(y^{-1}xs)\dd\lambda_G(y)\right)\overline{\xi(s)}\dd\lambda_N(s)
\\&=\int _Gf(y)\left(\int_N g(y^{-1}xs)\overline{\xi(s)}\dd\lambda_N(s)\right)\dd\lambda_G(y)
\\&=\int_Gf(y)T_\xi(g)(y^{-1}x)\dd\lambda_G(y)=f\ast_G T_\xi(g)(x),
\end{align*}
\end{proof}

\begin{corollary}
{\it Let $G$ be a locally compact group and $N$ be a closed subgroup of $G$.
The convolution $\ast_G:\mathcal{C}_c(G)\times \mathcal{C}_\xi(G,N)\to \mathcal{C}_\xi(G,N)$ given by $(f,\psi)\mapsto f\ast_G\psi$ is a left module action. 
}\end{corollary}
\begin{remark}\label{J1.M1}
Let $\xi=1$ be the trivial character of the normal subgroup $N$. Suppose $f\in\mathcal{C}_c(G)$ and $\psi\in \mathcal{C}_1(G,N)$ is identified with $\widetilde{\psi}\in\mathcal{C}_c(G/N)$ as in Remark \ref{J1}. Then $T_N(f)\ast_{G/N}\widetilde{\psi}$ identifies the convolution module action $f\ast_G\psi$. 
\end{remark}
\begin{remark}\label{NC}
Suppose that $N$ is a compact and normal subgroup of $G$. Let $f\in\mathcal{C}_c(G)$ and $\psi\in \mathcal{C}_\xi(G,N)$. Invoking Proposition 3.1(1) of \cite{AGHF.cov.com}, we obtain $\mathcal{C}_\xi(G,N)\subseteq\mathcal{C}_c(G)$ and hence the convolution module action $f\ast_G\psi$ is the standard convolution of functions in $\mathcal{C}_c(G)$.  
\end{remark}

We then continue by the following norm property of the convolution module action (\ref{f.psi}).

\begin{theorem}
Let $G$ be a locally compact group and $N$ be a closed normal subgroup of $G$. Suppose $\xi\in\chi(N)$ and $1\le p<\infty$. Let $f\in\mathcal{C}_c(G)$ and $\psi\in \mathcal{C}_\xi(G,N)$. Then  
\[
\|f\ast_G\psi\|_{(p)}\le\|f\|_{L^1(G)}\|\psi\|_{(p)}.
\]
\end{theorem}
\begin{proof}
Let $1\le p<\infty$, $f\in\mathcal{C}_c(G)$, and $\psi\in \mathcal{C}_\xi(G,N)$. Then, for any $x\in G$ 
\begin{equation}\label{|f.psi|}
|f\ast_G\psi(x)|\le\int_G|f(y)||\psi(y^{-1}x)|\dd\lambda_G(y).
\end{equation}
Using Minkowski's integral inequality, and (\ref{|f.psi|}) we have 
\begin{align*}
\|f\ast_G\psi\|_{(p)}&\le
\int_{G}\left(\int_{G/N}|f(y)|^p|\psi(y^{-1}x)|^p\dd\lambda_{G/N}(xN)\right)^{1/p}\dd\lambda_G(y)
\\&=\int_{G}|f(y)|\left(\int_{G/N}|\psi(y^{-1}x)|^p\dd\lambda_{G/N}(xN)\right)^{1/p}\dd\lambda_G(y)
\\&=\int_{G}|f(y)|\left(\int_{G/N}|\psi(x)|^p\dd\lambda_{G/N}(yxN)\right)^{1/p}\dd\lambda_G(y)=\|f\|_{L^1(G)}\|\psi\|_{(p)}.
\end{align*}
\end{proof}

\begin{corollary}
{\it Let $G$ be a locally compact group and $N$ be a closed normal subgroup of $G$. Then, 
\begin{enumerate}
\item $\mathcal{C}_\xi(G,N)$ is a normed left $\mathcal{C}_c(G)$-module with respect to the left module action (\ref{f.psi}). 
\item $T_\xi:\mathcal{C}_c(G)\to \mathcal{C}_\xi(G,N)$ is a $\mathcal{C}_c(G)$-module homomorphism.
\end{enumerate} 
}\end{corollary}

For every $1\le p<\infty$, suppose that $L^p_\xi(G,N)$ is the Banach completion of the normed linear space $\mathcal{C}_\xi(G,N)$ with respect to $\|\cdot\|_{(p)}$ given by (\ref{(p)}). We shall denote the completion norm by $\|\cdot\|_{(p)}$ as well. 

\begin{remark}\label{J1.Lp}
Suppose that $\xi=1$ is the trivial character of the normal subgroup $N$ and $1\le p<\infty$. 
Using the isometric identification of $\psi\in \mathcal{C}_1(G,N)$ with $\widetilde{\psi}\in\mathcal{C}_c(G/N)$ as in Remark \ref{J1.p}, we conclude that the Banach space $L^p_1(G,N)$ is isometrically isomorphic to $L^p(G/N)$. 
\end{remark}
\begin{remark}\label{NCp}
Suppose that $N$ is a compact and normal subgroup of $G$. Let $1\le p<\infty$ and $\xi\in\chi(N)$. Invoking Proposition 3.1(1) of \cite{AGHF.cov.com}, we obtain $\mathcal{C}_\xi(G,N)\subseteq\mathcal{C}_c(G)$. Then Proposition \ref{NC.norm.p} guarantess that $L^p_\xi(G,N)$ is a closed subspace of $L^p(G)$, see  \cite{AGHF.cov.com} for more details on structure of these Banach spaces.   
\end{remark}

\begin{theorem}\label{astLp}
{\it Let $G$ be a locally compact group and $N$ be a closed normal subgroup of $G$. Suppose $\xi\in\chi(N)$ and $1\le p<\infty$. The module action $\ast_G:\mathcal{C}_c(G)\times \mathcal{C}_\xi(G,N)\to \mathcal{C}_\xi(G,N)$ 
has a unique extension to a module action  
$\ast_G:L^1(G)\times L^p_\xi(G,N)\to L^p_\xi(G,N)$, 
in which the Banach space $L^p_\xi(G,N)$ equipped with the extended module action is a Banach left $L^1(G)$-module.
}
\end{theorem}
\begin{proof}
Let $f\in L^1(G)$ and $\psi\in L^p_\xi(G,N)$. Suppose $(f_n)\subset\mathcal{C}_c(G)$ and $(\psi_n)\subset \mathcal{C}_\xi(G,N)$ with $f=\lim_nf_n$ in $L^1(G)$ 
and $\psi=\lim_n\psi_n$ in $L^p_\xi(G,N)$. Then $(f_n\ast_G\psi_n)$ is a Cauchy sequence in $L^p_\xi(G,N)$. Since $L^p_\xi(G,N)$ is a Banach space, the sequence $(f_n\ast_G\psi_n)$ has a limit in $L^p_\xi(G,N)$. We then define 
$f\ast_G\psi:=\lim_nf_n\ast_G\psi_n$. Then, $f\ast_G\psi\in L^p_\xi(G,N)$.
One can also check that $f\ast_G\psi$ is independent of the choice of the sequences $(f_n)$ and $(\psi_n)$. Therefore, $(f,\psi)\mapsto f\ast_G\psi$ uniquely defines an extension of the left module action $\ast_G:\mathcal{C}_c(G)\times \mathcal{C}_\xi(G,N)\to \mathcal{C}_\xi(G,N)$ to the module action  
$\ast_G:L^1(G)\times L^p_\xi(G,N)\to L^p_\xi(G,N)$ which satisfies  
$\|f\ast_G\psi\|_{(p)}\le\|f\|_{L^1(G)}\|\psi\|_{(p)}$, for every $f\in L^1(G)$ and $\psi\in L^p_\xi(G,N)$. This implies that the Banach space $L^p_\xi(G,N)$ equipped with the extended module action $\ast_G:L^1(G)\times L^p_\xi(G,N)\to L^p_\xi(G,N)$ is a Banach left $L^1(G)$-module.
\end{proof}

\begin{remark}\label{func.real}
The Banach space $L^p_\xi(G,N)$ can be identified with a linear space of complex-valued functions on $G$, where two functions are identified when they are equal locally almost everywhere (l.c.a). Let $\mathfrak{A}\in L^p_\xi(G,N)$ be an arbitrary equivalent class of Cauchy sequences in $(\mathcal{C}_\xi(G,N),\|\cdot\|_{(p)})$ with a representative $(\psi_n)$. Then $(\psi_n)$ is a Cauchy sequence in $(\mathcal{C}_\xi(G,N),\|\cdot\|_{(p)})$ and $\|\mathfrak{A}\|_{(p)}^p=\lim_n\|\psi_n\|_{(p)}^p$. Let $(\psi_{\epsilon(n)})$ be a subsequence of $(\psi_n)$ such that $\|\psi_{\epsilon(n+1)}-\psi_{\epsilon(n)}\|_{(p)}<1/2^n$, for every $n\in \mathbb{N}$. So $\sum_{k=1}^\infty\|\psi_{\epsilon(k+1)}-\psi_{\epsilon(k)}\|_{(p)}<\infty$.
Therefore, $\sum_{k=1}^\infty\|\Psi_k\|_{L^p(G/N)}<\infty$ if $\Psi_k:=|\psi_{\epsilon(k+1)}-\psi_{\epsilon(k)}|$ for $k\in\mathbb{N}$. Since $L^p(G/N)$ is a Banach space, there exists $\Psi\in L^p(G/N)$ with $\Psi=\sum_{k=1}^\infty\Psi_k$ in $L^p(G/N)$, where the sum also converges pointwise almost everywhere. Hence, $
\sum_{k=1}^\infty\Psi_k(xN)<\infty$, for every $xN\in G/N$ except those in a subset set $E$ of $G/N$ with $\lambda_{G/N}(E)=0$.
So the series $\sum_{k=1}^\infty\psi_{\epsilon(k+1)}(x)-\psi_{\epsilon(k)}(x),$ converges to a complex number, denoted by $\psi_0(x)$, for every $x\in G$ with $xN\not\in E$. For $x\in G$ with $xN\notin E$, let $\psi(x):=\psi_{\epsilon(1)}(x)+\psi_0(x)$. Invoking Theorem 3.3.28 of \cite{50}, $\psi(x)$ is independent of the chosen $(\psi_n)$ as a representative of the class $\mathfrak{A}$. So the equivalence class $\mathfrak{A}$ can be uniquely determined with the locally almost everywhere defined function $x\mapsto\psi(x)$, as a complex-valued function on $G$. 
We then obtain   
\[
\lim_{n}\psi_{\epsilon(n)}(x)=\lim_n\left(\psi_{\epsilon(1)}(x)+\sum_{k=1}^{n-1}\psi_{\epsilon(k+1)}(x)-\psi_{\epsilon(k)}(x)\right)=\psi(x).
\]
Under this identification, for every $n\in\mathbb{N}$ and $x\in G$, we have   
\[
|\psi_{\epsilon(n)}(x)|\le|\psi_{\epsilon(1)}(x)|+\sum_{k=1}^{n-1}\Psi_k(x)\le |\psi_{\epsilon(1)}(x)|+\Psi(x),
\]
implying that the function $xN\mapsto |\psi(x)|$ belongs to $L^p(G/N)$ and 
\begin{equation}
\|\mathfrak{A}\|_{(p)}^p=\lim_{n}\|\psi_{\epsilon(n)}\|_{(p)}^p=\int_{G/N}|\psi(x)|^p\dd\lambda_{G/N}(xN).
\end{equation}
\end{remark}

It is proved that the linear operator $T_\xi:(\mathcal{C}_c(G),\|.\|_{L^1(G)})\to(\mathcal{C}_\xi(G,N),\|.\|_{(1)})$ is a contraction, see Theorem 4.4 of \cite{AGHF.ch.normal.main}. Invoking the structure of the Banach space  $L^1_\xi(G,N)$, the linear operator  
$T_\xi:(\mathcal{C}_c(G),\|.\|_{L^1(G)})\to(\mathcal{C}_\xi(G,N),\|.\|_{(1)})$ 
has a unique extension to a contraction from $L^1(G)$ onto $L^1_\xi(G,N)$. 
The unique extension is denoted by $T_\xi:L^1(G)\to L^1_\xi(G,N)$. 
It is shown that 
\[
T_\xi(f)(x)=\int_Nf(xs)\overline{\xi(s)}\dd\lambda_N(s)\ \ \ \ {\rm for\ \ l.a.e.}\ x\in G,
\]
for every $f\in L^1(G)$, see Theorem 5.3 of \cite{AGHF.ch.normal.main}.

We then have the following explicit formula for the extension of the convolution module action (\ref{f.psi}) according to Theorem \ref{astLp}.

\begin{theorem}
Let $G$ be a locally compact group and $N$ be a closed normal subgroup of $G$.
Suppose $\xi\in\chi(N)$ and $1\le p<\infty$. Let $f\in L^1(G)$ and $\psi\in L^p_\xi(G,N)$. Then
\begin{equation}\label{f*psiLpFormula}
f\ast_G\psi(x)=\int_Gf(y)\psi(y^{-1}x)\dd\lambda_G(y)\ \ \ \ {\rm for\ \ l.a.e.}\ x\in G.
\end{equation}
\end{theorem}
\begin{proof}
Let $f\in L^1(G)$ and $\psi\in L^p_\xi(G,N)$. 
Suppose $(f_n)\subset\mathcal{C}_c(G)$ and $(\psi_n)\subset \mathcal{C}_\xi(G,N)$ with $\|f_n-f\|_{L^1(G)}<2^{-(n+2)}$ and $\|\psi_n-\psi\|_{(p)}<2^{-(n+2)}$ for $n\in\mathbb{N}$. We then have $f\ast_G\psi=\lim_nf_n\ast_G\psi_n$ in $L^p_\xi(G,N)$. Suppose $M>0$ with $\|f_n\|_{L^1(G)}<M$ and $\|\psi_n\|_{(p)}<M$ for every $n\in\mathbb{N}$. Then $\|f_{n+1}\ast\psi_{n+1}-f_n\ast\psi_n\|_{(p)}<M2^{-n}$ for $n\in\mathbb{N}$. Therefore, we get $f\ast_G\psi(x)=\lim_nf_n\ast_G\psi_n(x)$ for l.a.e. $x\in G$, according to Remark \ref{func.real}. 
Then,  
\begin{align*}
&\left(\int_{G/N}\left(\int_G|f_{n}(y)\psi_{n}(y^{-1}x)-f(y)\psi(y^{-1}x)|\dd\lambda_G(y)\right)^p\dd\lambda_{G/N}(xN)\right)^{1/p}
\\&\le\int_G\left(\int_{G/N}|f_{n}(y)\psi_{n}(y^{-1}x)-f(y)\psi(y^{-1}x)|^p\dd\lambda_{G/N}(xN)\right)^{1/p}\dd\lambda_G(y)
%\\&\le\int_G\left(\int_{G/N}|f_{n}(y)\psi_{n}(y^{-1}x)-f_{n}(y)\psi(y^{-1}x)+f_{n}(y)\psi(y^{-1}x)-f(y)\psi(y^{-1}x)|^p\dd\lambda_{G/N}(xN)\right)^{1/p}\dd\lambda_G(y)
%\\&\le\int_G\left(\int_{G/N}|f_{n}(y)\psi_{n}(y^{-1}x)-f_{n}(y)\psi(y^{-1}x)|^p\dd\lambda_{G/N}(xN)\right)^{1/p}\dd\lambda_G(y)
%\\&\ +\int_G\left(\int_{G/N}|f_{n}(y)\psi(y^{-1}x)-f(y)\psi(y^{-1}x)|^p\dd\lambda_{G/N}(xN)\right)^{1/p}\dd\lambda_G(y)
%\\&\le\int_G|f_n(y)|\left(\int_{G/N}|\psi_{n}(y^{-1}x)-\psi(y^{-1}x)|^p\dd\lambda_{G/N}(xN)\right)^{1/p}\dd\lambda_G(y)
%\\&\ +\int_G|f_{n}(y)-f(y)|\left(\int_{G/N}|\psi(y^{-1}x)|^p\dd\lambda_{G/N}(xN)\right)^{1/p}\dd\lambda_G(y)
\\&\le \|f_{n}\|_{L^1(G)}\|\psi_{n}-\psi\|_{(p)}+\|f_{n}-f\|_{L^1(G)}\|\psi\|_{(p)}<M2^{-(n+1)},
\end{align*}
which implies that  
\[
\int_{G/N}\left(\int_G|f_{n}(y)\psi_{n}(y^{-1}x)-f(y)\psi(y^{-1}x)|\dd\lambda_G(y)\right)^p\dd\lambda_{G/N}(xN)<M^p2^{-p(n+1)},
\]
for $n\in\mathbb{N}$. Therefore, 
\[
\lim_n\left(\int_G|f_{n}(y)\psi_{n}(y^{-1}x)-f(y)\psi(y^{-1}x)|\dd\lambda_G(y)\right)^p=0,
\]
for a.e. $xN\in G/N$. Since 
\[
\left|f_{n}\ast_G\psi_{n}(x)-\int_Gf(y)\psi(y^{-1}x)\dd\lambda_G(y)\right|\le\int_G|f_{n}(y)\psi_{n}(y^{-1}x)-f(y)\psi(y^{-1}x)|\dd\lambda_G(y),
\]
we conclude that  
\[
\lim_n\left|f_{n}\ast_G\psi_{n}(x)-\int_Gf(y)\psi(y^{-1}x)\dd\lambda_G(y)\right|=0,
\]
for a.e. $xN\in G/N$. This implies that the right hand side of  (\ref{f*psiLpFormula}) is well-defined as a function on $G$, for l.a.e $x\in G$. We also obtain 
\[
\lim_nf_{n}\ast_G\psi_{n}(x)=\int_Gf(y)\psi(y^{-1}x)\dd\lambda_G(y),
\]
for l.c.a $x\in G$. Then, for l.c.a $x\in G$, we have 
\[
f\ast_G\psi(x)=\lim_nf_{n}\ast_G\psi_{n}(x)=\int_Gf(y)\psi(y^{-1}x)\dd\lambda_G(y).
\]
\end{proof}

We then show that $T_\xi:L^1(G)\to L^1_\xi(G,N)$ is a module homomorphism.  

\begin{proposition}
{\it Let $G$ be a locally compact group, $N$ be a closed normal subgroup of $G$, and $\xi\in\chi(N)$. Then, $T_\xi:L^1(G)\to L^1_\xi(G,N)$ is a left Banach $L^1(G)$-module homomorphism.
}\end{proposition}
\begin{proof}
Let $f\in\mathcal{C}_c(G)$ and $g\in L^1(G)$. Suppose $(g_n)\subset\mathcal{C}_c(G)$ with $g=\lim_n g_n$ in $L^1(G)$. Using boundedness of the linear operator $T_\xi:L^1(G)\to L^1_\xi(G,N)$, continuity of the module action in each argument, and Theorem \ref{mult.Cc}, we get 
\begin{align*}
T_\xi(f\ast_G g)=T_\xi(\lim_nf\ast_Gg_n)=\lim_nT_\xi(f\ast_Gg_n)
=\lim_nf\ast_GT_\xi(g_n)=f\ast_GT_\xi(g).
\end{align*}
Similar method guarantees that $T_\xi(f\ast_G g)=f\ast_GT_\xi(g)$, if $f\in L^1(G)$.
\end{proof}
\begin{remark}
Let $\xi=1$ be the trivial character of the normal subgroup $N$ and $1\le p<\infty$. Suppose $f\in L^1(G)$ and $\psi\in L_1^p(G,N)$ is identified with $\widetilde{\psi}\in L^p(G/N)$ as in Remark \ref{J1.Lp}. Then Remark \ref{J1.M1} implies that the convolution module action of $f$ on $\psi$ can be identified by $T_N(f)\ast_{G/N}\widetilde{\psi}$. Therefore, convolution module action of $L^1(G)$ on $L_1^p(G,N)$ coincides with the classical convolution module action of $L^1(G/N)$ on $L^p(G/N)$.
\end{remark}
\begin{remark}\label{NC.L1.Lp}
Suppose that $N$ is a compact and normal subgroup of $G$. Let $1\le p<\infty$ and $\xi\in\chi(N)$. Invoking Remarks \ref{NC} and \ref{NCp} we conclude that the convolution module action $f\ast_G\psi$ is the standard convolution module action of $L^1(G)$ on $L^p(G)$ given by (\ref{f.ast.g}).
\end{remark}

\section{\bf Convolution Modules of Semi-direct Product Group Algebras on Covariant Functions}
Throughout this section, we study different aspects of the structure of convolution module actions induced by group algebras on covariant functions of characters of canonical normal subgroups in semi-direct product groups, for more details on harmonic analysis of semi-direct product groups we refer the reader to see \cite[\S15.26 and \S15.29]{HR1}. 

Suppose that $H,K$ are locally compact groups and $\theta:H\to Aut(K)$ is a continuous homomorphism. Let $G_\theta=H\ltimes_\theta K$ be the semi-direct product of $H$ and $K$ with respect to $\theta$. 
The semi-direct product $G_\theta=H\ltimes_\theta K$ is the locally compact topological group with the underlying set $H\times K$
which is equipped by the product topology and the group operation 
\begin{equation*}
(h,k)\ltimes_\theta(h',k')=(hh',k\theta_h(k'))\hspace{0.5cm}{\rm and}\hspace{0.5cm}(h,k)^{-1}=(h^{-1},\theta_{h^{-1}}(k^{-1})).
\end{equation*}
Assume that $N$ is a closed normal subgroup of $K$ such that $\theta_h(N)=N$, for every $h\in H$. For $h\in H$, let $\delta_{H,N}(h)\in(0,\infty)$ be given by 
\begin{equation}\label{delta.H.N}
\dd\lambda_N(\theta_h(s))=\delta_{H,N}(h)^{-1}\dd\lambda_N(s),
\end{equation}
where $\lambda_N$ is a left Haar measure on $N$.
The continuous homomorphism $\delta_{H,N}$ characterizes Haar measures on the semi-direct product group $G_\theta$.
If $\lambda_H$ and $\lambda_K$ are left Haar measures on $H$ and $K$ respectively. Then, $\dd\lambda_{G_\theta}(h,k):=\delta_{H,K}(h)\dd\lambda_H(h)\dd\lambda_K(k)$, 
is a left Haar measure on $G_\theta$.

The continuous homomorphism $\theta:H\to Aut(K)$ induces the continuous homomorphism of $H$ into $Aut(K/N)$, denoted by $\widetilde{\theta}:H\to Aut(K/N)$, defined by $h\mapsto\widetilde{\theta}_h$, where $\widetilde{\theta}_h:K/N\to K/N$ is given by $\widetilde{\theta}_h(kN):=\theta_h(k)N$, for every $h\in H$ and $kN\in K/N$. Then the factor group $G_\theta/N$ and the semi-direct product group $G_{\widetilde{\theta}}:=H\ltimes_{\widetilde{\theta}}(K/N)$ are canonically isomorphic as topological groups via the topological group isomorphism $(h,k)N\mapsto (h,kN)$. 
\begin{proposition}
{\it Suppose $H,K$ are locally compact groups and $\theta:H\to Aut(K)$ is  a continuous homomorphism. Let $N$ be a closed and normal subgroup of $K$ such that $\theta_h(N)=N$ for every $h\in H$. Assume that $\lambda_{K/N}$ is the left Haar measure on $K/N$ normalized with respect to  left Haar measures $\lambda_N$ on $N$ and $\lambda_K$ on $K$. Then
\begin{enumerate}
\item $\delta_{H,K}(h)=\delta_{H,N}(h)\delta_{H,K/N}(h)$ for every $h\in H$.
\item $\dd\lambda_{G_{\widetilde{\theta}}}(h,kN):=\delta_{H,K/N}(h)\dd\lambda_{H}(h)\dd\lambda_{K/N}(kN)$ uniquely identifies the left Haar measure on $G_\theta/N$ normalized with respect to left Haar measures $\lambda_{G_\theta}$ on $G_\theta$ and $\lambda_N$ on $N$.
\end{enumerate}
}\end{proposition}
\begin{proof}
(1) This follows from Proposition 11 of \cite[Chap. VII, \S2]{Bour}. \\
(2) Let $\Phi\in\mathcal{C}_c(G_\theta/N)$ and $f\in\mathcal{C}_c(G_\theta)$ with $T_N(f)=\Phi$. Suppose that $\lambda_{G_\theta/N}$ is the left Haar measure on $G_\theta/N$ normalized with respect to left Haar measures $\lambda_{G_\theta}$ on $G_\theta$ and $\lambda_N$ on $N$. Then 
\begin{align*} 
\int_{G_\theta/N}\Phi(xN)\dd\lambda_{G_\theta/N}(xN)
%&=\int_{G_\theta}f(h,k)\dd\lambda_{G_\theta}(h,k)
&=\int_{H}\int_Kf(h,k)\delta_{H,K}(h)\dd\lambda_H(h)\dd\lambda_K(k)
\\&=\int_H\int_{K/N}\left(\int_Nf(h,ks)\dd\lambda_N(s)\right)\dd\lambda_{K/N}(kN)\delta_{H,K}(h)\dd\lambda_H(h)
\\&=\int_H\int_{K/N}\left(\int_Nf(h,ks)\dd\lambda_N(s)\right)\dd\lambda_{K/N}(kN)\delta_{H,N}(h)\delta_{H,K/N}(h)\dd\lambda_H(h)
\\&=\int_H\int_{K/N}\left(\int_Nf(h,ks)\dd\lambda_N(\theta_{h^{-1}}(s))\right)\dd\lambda_{K/N}(kN)\delta_{H,K/N}(h)\dd\lambda_H(h)
\\&=\int_H\int_{K/N}\left(\int_Nf(h,k\theta_h(s))\dd\lambda_N(s)\right)\dd\lambda_{K/N}(kN)\delta_{H,K/N}(h)\dd\lambda_H(h)
\\&=\int_H\int_{K/N}\Phi((h,k)N)\dd\lambda_{K/N}(kN)\delta_{H,K/N}(h)\dd\lambda_H(h).
\end{align*}
\end{proof}

Suppose that $N$ is a closed subgroup of $K$ with $\theta_h(N)=N$ for every $h\in H$. For $\xi\in\chi(N)$ and $\psi\in \mathcal{C}_\xi(G_\theta,N)$, we have  
\begin{equation}\label{main.cov.prop}
\psi(h,s)=\xi\circ\theta_{h^{-1}}(s)\psi(h,e_K)\ \ \ \ \ \ {\rm for}\ \ (h,s)\in H\times N.
\end{equation}
In this case, for $1\le p<\infty$, the norm $\|\psi\|_{(p)}$ can be computed by  
\[
\|\psi\|_{(p)}^p:=\int_H\int_{K/N}|\psi(h,k)|^p\delta_{H,K/N}(h)\dd\lambda_{H}(h)\dd\lambda_{K/N}(kN).
\]
Let $f\in\mathcal{C}_c(G_\theta)$, $\psi\in \mathcal{C}_\xi(G_\theta,N)$, and $(a,b)\in G_\theta$. Then $f\ast\psi:G_\theta\to\mathbb{C}$ is given by  
\begin{equation}\label{*GtheaN}
f\ast\psi(a,b)=\int_H\int_{K}f(h,k)\psi(h^{-1}a,\theta_{h^{-1}}(k^{-1}b))\delta_{H,K}(h)\dd\lambda_{H}(h)\dd\lambda_{K}(k).
\end{equation}

We then have the following explicit formula for $f\ast\psi$, when $N=K$.

\begin{proposition}\label{Jxi.K}
{\it Let $H,K$ be locally compact groups and $\theta:H\to Aut(K)$ be a continuous homomorphism and $G_\theta=H\ltimes_\theta K$. Suppose $\xi\in\chi(K)$, $f\in\mathcal{C}_c(G_\theta)$, $\psi\in\mathcal{C}_\xi(G_\theta,K)$, and $(a,b)\in G_\theta$. Then 
\[
f\ast\psi(a,b)=\xi\circ\theta_{a^{-1}}(b)\int_H\int_{K}f(h,k)\psi(h^{-1}a,e_K)\overline{\xi\circ\theta_{a^{-1}}(k)}\delta_{H,K}(h)\dd\lambda_{H}(h)\dd\lambda_{K}(k).
\]
}\end{proposition}
\begin{proof}
Let $f\in\mathcal{C}_c(G_\theta)$, $\psi\in\mathcal{C}_\xi(G_\theta,K)$ and $(a,b)\in G_\theta$. Using (\ref{main.cov.prop}) and (\ref{*GtheaN}), we get 
\begin{align*}
f\ast\psi(a,b)
&=\int_H\int_{K}f(h,k)\psi(h^{-1}a,\theta_{h^{-1}}(k^{-1}b))\delta_{H,K}(h)\dd\lambda_{H}(h)\dd\lambda_{K}(k)
\\&=\int_H\int_{K}f(h,k)\xi\circ\theta_{a^{-1}h}(\theta_{h^{-1}}(k^{-1}b))\psi(h^{-1}a,e_K)\delta_{H,K}(h)\dd\lambda_{H}(h)\dd\lambda_{K}(k)
\\&=\int_H\int_{K}f(h,k)\xi\circ\theta_{a^{-1}}(k^{-1}b)\psi(h^{-1}a,e_K)\delta_{H,K}(h)\dd\lambda_{H}(h)\dd\lambda_{K}(k)
\\&=\xi\circ\theta_{a^{-1}}(b)\int_H\int_{K}f(h,k)\xi\circ\theta_{a^{-1}}(k^{-1})\psi(h^{-1}a,e_K)\delta_{H,K}(h)\dd\lambda_{H}(h)\dd\lambda_{K}(k)
\\&=\xi\circ\theta_{a^{-1}}(b)\int_H\int_{K}f(h,k)\psi(h^{-1}a,e_K)\overline{\xi\circ\theta_{a^{-1}}(k)}\delta_{H,K}(h)\dd\lambda_{H}(h)\dd\lambda_{K}(k).
\end{align*}
\end{proof}

We then finish the paper by investigation of convolution module actions induced by group algebras on covariant functions of characters in different examples of semi-direct product groups. 

\subsection{Abstract Weyl-Heisenberg Groups}
Let $L$ be a locally compact Abelian group and $\widehat{L}$ be the dual group of $L$. Assume that $H:=L$ and $K:=\widehat{L}\times\mathbb{T}$. Suppose that $\theta:H\to Aut(K)$ is the continuous homomorphism $x\mapsto\theta_x$, where $\theta_x:K\to K$ is defined by $\theta_x(\omega,z):=(\omega,\omega(x)z)$,
for $(\omega,z)\in K=\widehat{L}\times\mathbb{T}$.
The abstract Weyl-Heisenberg group associated to the LCA group $L$, denoted by $\mathbb{W}(L)$, is the semi-direct product $H\ltimes_\theta K$, 
see \cite{ AGHF.GWHG}. 
\subsubsection{\bf The Case $N=\mathbb{T}$}
Let $N:=\{(e_L,\mathbf{1},z):z\in\mathbb{T}\}$. Then $N$ is a closed and central subgroup of $\mathbb{W}(L)$. Also, $\theta_x(\mathbf{1},z)=(\mathbf{1},z)$ for $x\in L$ and $z\in\mathbb{T}$. Then $\chi(N)=\{\chi_n:n\in \mathbb{Z}\}$, where $\chi_n:\mathbb{T}\to\mathbb{T}$ is defined by $\chi_n(z):=z^n$ for $z\in\mathbb{T}$ and $n\in\mathbb{Z}$. Let $n\in\mathbb{Z}$ and $\xi:=\chi_n$.
Then, every $\psi\in\mathcal{C}_\xi(\mathbb{W}(L),\mathbb{T})$ satisfies 
\begin{equation}\label{cov.prop.T}
\psi(x,\omega,z)=z^n\psi(x,\omega,1)\ \ \ \ {\rm for}\ (x,\omega,z)\in\mathbb{W}(L).
\end{equation}
In this case, for $1\le p<\infty$, the norm $\|\psi\|_{(p)}$ is given by  
\[
\|\psi\|_{(p)}^p=\int_L\int_{\widehat{L}}|\psi(x,\omega,1)|^p\dd\lambda_{L}(x)\dd\lambda_{\widehat{L}}(\omega)=\|\psi\|_{L^p(\mathbb{W}(L))}^p.
\]
Let $f\in\mathcal{C}_c(\mathbb{W}(L))$, $\psi\in \mathcal{C}_\xi(\mathbb{W}(L),\mathbb{T})$ and $(x,\omega,z)\in\mathbb{W}(L)$. Then, using (\ref{*GtheaN}), we obtain   
\begin{align*}
f\ast\psi(x,\omega,z)
=\frac{1}{2\pi}\int_L\int_{\widehat{L}}\int_0^{2\pi}f(x',\omega',e^{\ii\alpha})\psi(x-x',\omega\overline{\omega'},\overline{\omega(x')}\omega'(x')ze^{-\ii\alpha})\dd\alpha\dd\lambda_{L}(x')\dd\lambda_{\widehat{L}}(\omega').
\end{align*}
Then, (\ref{cov.prop.T}) implies that  
\begin{equation}\label{f*psi.T}
f\ast\psi(x,\omega,z)
=\frac{z^n}{2\pi}\int_L\int_{\widehat{L}}\int_0^{2\pi}\omega(x')^{-n}\omega'(x')^ne^{-\ii n\alpha}f(x',\omega',e^{\ii\alpha})\psi(x-x',\omega\overline{\omega'},1)\dd\alpha\dd\lambda_{L}(x')\dd\lambda_{\widehat{L}}(\omega').
\end{equation}
\begin{example}
Let $d\ge 1$ and $L:=\mathbb{R}^d$. Then $\widehat{\mathbb{R}^d}=\{\widehat{\bom}:\bom\in\mathbb{R}^d\}$, where $\widehat{\bom}:\mathbb{R}^d\to\mathbb{T}$ is given by $\widehat{\bom}(\mathbf{x}):=e^{\ii\langle\mathbf{x},\bom\rangle}$, for all $\mathbf{x}\in\mathbb{R}^d$. 
Let $n\in\mathbb{Z}$, $\xi:=\chi_n$, and $\psi\in \mathcal{C}_{\xi}(\mathbb{W}(\mathbb{R}^d),\mathbb{T})$. Suppose $f\in\mathcal{C}_c(\mathbb{W}(\mathbb{R}^d))$ and $(\mathbf{x},\bom,z)\in\mathbb{W}(\mathbb{R}^d)$. Then, using (\ref{f*psi.T}), we get 
\[
f\ast\psi(\mathbf{x},\bom,z)=\frac{z^n}{2\pi}\int_{\mathbb{R}^d}\int_{\mathbb{R}^d}\int_0^{2\pi}e^{\ii n\langle\bom'-\bom,\mathbf{x}'\rangle}e^{-\ii n\alpha}f(\mathbf{x}',\bom',e^{\ii\alpha})\psi(\mathbf{x}-\mathbf{x}',\bom-\bom',1)\dd\alpha\dd\mathbf{x'}\dd\bom'.
\] 
\end{example}
\begin{example}
Let $d\ge 1$ and $L:=\mathbb{Z}^d$. Then 
$\widehat{\mathbb{Z}^d}=\{\widehat{\mathbf{z}}:\mathbf{z}=(z_1,...,z_d)^T\in\mathbb{T}^d\}$, where the character $\widehat{\mathbf{z}}:\mathbb{Z}^d\to\mathbb{T}$ is given by 
$\widehat{\mathbf{z}}(\mathbf{k}):=\prod_{j=1}^dz_j^{k_j}$, for $\mathbf{k}:=(k_1,...,k_d)^T\in\mathbb{Z}^d$ and $\mathbf{z}:=(z_1,...,z_d)^T\in\mathbb{T}^d$. Let $n\in\mathbb{Z}$, $\xi:=\chi_n$, and $\psi\in \mathcal{C}_{\xi}(\mathbb{Z}^d\ltimes\mathbb{T}^{d+1},\mathbb{T})$. Suppose $f\in\mathcal{C}_c(\mathbb{Z}^d\ltimes\mathbb{T}^{d+1})$ and $(\mathbf{k},\mathbf{z},z)\in\mathbb{W}(\mathbb{Z}^d)=\mathbb{Z}^d\ltimes\mathbb{T}^{d+1}$. Then, using (\ref{f*psi.T}), we have 
\[
f\ast\psi(\mathbf{k},\mathbf{z},z)=\frac{z^n}{2\pi}\sum_{\mathbf{k}'\in\mathbb{Z}^d}\int_{\mathbb{T}^d}\int_0^{2\pi}\widehat{\mathbf{z}}(\mathbf{k}')^{-n}\widehat{\mathbf{z}'}(\mathbf{k}')^ne^{-\ii n\alpha}f(\mathbf{k}',\mathbf{z}',e^{\ii\alpha})\psi(\mathbf{k}-\mathbf{k}',\mathbf{z}\overline{\mathbf{z}'},1)\dd\alpha\dd\mathbf{z}'.
\]
\end{example}
\begin{example}
Let $M>0$ be a positive integer and $L:=\mathbb{Z}_M$ be the finite cyclic 
group of integers modulo $M$. Then $\widehat{L}=\{\widehat{\ell}:\ell\in\mathbb{Z}_M\}$, where $\widehat{\ell}:\mathbb{Z}_M\to\mathbb{T}$ is given by 
$\widehat{\ell}(m):=e^{2\pi\ii m\ell/M}$, for $m,\ell\in\mathbb{Z}_M$.  Let $n\in\mathbb{Z}$, $\xi:=\chi_n$, and $\psi\in \mathcal{C}_\xi(\mathbb{Z}_M\ltimes(\mathbb{Z}_M\times\mathbb{T}),\mathbb{T})$. Suppose $f\in\mathcal{C}_c(\mathbb{W}(\mathbb{Z}_M))$ and $(m,\ell,z)\in\mathbb{W}(\mathbb{Z}_M)$. Then, using (\ref{f*psi.T}), we achieve 
\[
f\ast\psi(m,\ell,z)=\frac{z^n}{2\pi}\sum_{m'=0}^{M-1}\sum_{\ell'=0}^{M-1}\int_0^{2\pi}e^{2\pi\ii n(\ell'-\ell)m'/M}e^{-\ii n\alpha}f(m',\ell',e^{\ii\alpha})\psi(m-m',\ell-\ell',1)\dd\alpha.
\]
\end{example}
\subsubsection{\bf The Case $N=\widehat{L}\times\mathbb{T}$}
Let $N:=\{(e_L,\omega,z):\omega\in\widehat{L},z\in\mathbb{T}\}$.
Then, $N$ is normal in $\mathbb{W}(L)$ and $\chi(N)=\{\chi_{y,n}:y\in L, n\in\mathbb{Z}\}$, where $\chi_{y,n}:N\to\mathbb{T}$ is defined by 
$\chi_{y,n}(\zeta,e^{\ii\alpha}):=\zeta(y)e^{\ii n\alpha}$, for $(\zeta,e^{\ii\alpha})\in\widehat{L}\times\mathbb{T}$. 
Let $\xi:=\chi_{y,n}$ with $y\in L$ and $n\in\mathbb{N}$. 
Then, each $\psi\in \mathcal{C}_\xi(\mathbb{W}(L),\widehat{L}\times\mathbb{T})$ satisfies 
\begin{equation}\label{cov.prop.LT}
\psi(x,\omega,z)=\omega(y)\omega(x)^{-n}z^n\psi(x,\mathbf{1},1)\ \ \ \ {\rm for}\ \ (x,\omega,z)\in\mathbb{W}(L).
\end{equation}
In this case, for $1\le p<\infty$, the norm $\|\psi\|_{(p)}$ is given by  
\[
\|\psi\|_{(p)}^p=\int_L|\psi(x,\mathbf{1},1)|^p\dd\lambda_{L}(x).
\]
Let $f\in\mathcal{C}_c(\mathbb{W}(L))$, $\psi\in \mathcal{C}_\xi(\mathbb{W}(L),\widehat{L}\times\mathbb{T})$ and $(x,\omega,z)\in\mathbb{W}(L)$. Then   
\begin{align*}
f\ast\psi(x,\omega,z)
=\frac{1}{2\pi}\int_L\int_{\widehat{L}}\int_0^{2\pi}f(x',\omega',e^{\ii\alpha})\psi(x-x',\omega\overline{\omega'},\overline{\omega(x')}\omega'(x')ze^{-\ii\alpha})\dd\alpha\dd\lambda_{L}(x')\dd\lambda_{\widehat{L}}(\omega'),
\end{align*}
and applying (\ref{cov.prop.LT}), we get 
\begin{equation}\label{f*psi.LT}
f\ast\psi(x,\omega,z)
=\frac{z^n\omega(y)}{2\pi\omega(x)^n}\int_L\int_{\widehat{L}}\int_0^{2\pi}f(x',\omega',e^{\ii\alpha})\psi(x-x',\mathbf{1},1)\overline{\omega'(y)}\omega'(x)^{n}e^{-\ii n\alpha}\dd\alpha\dd\lambda_{L}(x')\dd\lambda_{\widehat{L}}(\omega').
\end{equation}
\begin{example}
Let $d\ge 1$ and $L:=\mathbb{R}^d$. Assume that $(\mathbf{y},n)\in\mathbb{R}^d\times\mathbb{Z}$ and $\xi:=\chi_{\mathbf{y},n}$. Suppose $\psi\in \mathcal{C}_{\xi}(\mathbb{W}(\mathbb{R}^d),\mathbb{R}^d\times\mathbb{T})$, $f\in\mathcal{C}_c(\mathbb{R}^{2d}\times\mathbb{T})$, and $(\mathbf{x},\bom,z)\in\mathbb{W}(\mathbb{R}^d)$. Then, using (\ref{f*psi.LT}), we have
\[
f\ast\psi(\mathbf{x},\bom,z)=\frac{z^ne^{\ii\langle\mathbf{y},\bom\rangle}}{2\pi e^{\ii n\langle\mathbf{x},\bom\rangle}}\int_{\mathbb{R}^d}\int_{\mathbb{R}^d}\int_0^{2\pi}f(\mathbf{x}',\bom',e^{\ii\alpha})\psi(\mathbf{x}-\mathbf{x}',\mathbf{0},1)e^{-\ii\langle\bom',\mathbf{y}-n\mathbf{x}\rangle}e^{-\ii n\alpha}\dd\alpha\dd\mathbf{x'}\dd\bom'.
\] 
\end{example}
\begin{example}
Let $d\ge 1$ and $L:=\mathbb{Z}^d$. Assume that $(\mathbf{n},n)\in\mathbb{Z}^d\times\mathbb{Z}$ and $\xi:=\chi_{\mathbf{n},n}$. Suppose $\psi\in \mathcal{C}_\xi(\mathbb{Z}^d\ltimes\mathbb{T}^{d+1},\mathbb{T}^{d+1})$, $f\in\mathcal{C}_c(\mathbb{Z}^d\ltimes\mathbb{T}^{d+1})$, and $(\mathbf{k},\mathbf{z},z)\in\mathbb{W}(\mathbb{Z}^d)$. Then, using (\ref{f*psi.LT}), we get 
\[
f\ast\psi(\mathbf{k},\mathbf{z},z)=\frac{z^n\widehat{\mathbf{z}}(\mathbf{n})}{2\pi\widehat{\mathbf{z}}(\mathbf{k})^n}\sum_{\mathbf{k}'\in\mathbb{Z}^d}\int_{\mathbb{T}^d}\int_0^{2\pi}f(\mathbf{k}',\mathbf{z}',e^{\ii\alpha})\psi(\mathbf{k}-\mathbf{k}',\mathbf{1},1)\overline{\widehat{\mathbf{z}'}(\mathbf{n})}\widehat{\mathbf{z}'}(\mathbf{k})^ne^{-\ii n\alpha}\dd\alpha\dd\mathbf{z}'.
\]
\end{example}
\begin{example}
Let $M>0$ be a positive integer and $L:=\mathbb{Z}_M$ be the finite cyclic 
group of integers modulo $M$. Assume that $(k,n)\in\mathbb{Z}_M\times\mathbb{Z}$ and $\xi:=\chi_{k,n}$. Suppose $\psi\in \mathcal{C}_\xi(\mathbb{W}(\mathbb{Z}_M),\mathbb{Z}_M\times\mathbb{T})$, $f\in\mathcal{C}_c(\mathbb{W}(\mathbb{Z}_M))$, and $(m,\ell,z)\in\mathbb{W}(\mathbb{Z}_M)$. Then, using (\ref{f*psi.LT}), we achieve 
\[
f\ast\psi(m,\ell,z)=\frac{z^ne^{2\pi\ii\ell k/M}}{2\pi e^{2\pi\ii nmk/M}}\sum_{m'=0}^{M-1}\sum_{\ell'=0}^{M-1}\int_0^{2\pi}f(m',\ell',e^{\ii\alpha})\psi(m-m',\ell-\ell',1)e^{-2\pi\ii\ell'(k-nm)/M}e^{-\ii n\alpha}\dd\alpha.
\]
\end{example}

\subsection{The Heisenberg Groups}
Assume that $d\ge 1$ and locally compact groups $H,K$ are given by $H:=\mathbb{R}^d$ and $K:=\mathbb{R}^d\times\mathbb{R}$. 
For each $\mathbf{x}\in H$, define the map $\theta_\mathbf{x}:K\to K$ by 
$\theta_\mathbf{x}(\mathbf{y},s):=(\mathbf{y},s+\langle\mathbf{x},\mathbf{y}\rangle)$, for all $(\mathbf{y},s)\in K$. The Heisenberg group $\mathbb{H}^d:=\mathbb{R}^{2d+1}$ is the semi-direct product $G_\theta:=H\ltimes_\theta K$. Then $\chi(K)=\{\mathbf{e}_{\mathbf{z},\nu}:\nu\in\mathbb{R},\mathbf{z}\in\mathbb{R}^d\}$, where 
$\mathbf{e}_{\mathbf{z},\nu}(\mathbf{y},s)
:=e^{\ii\langle\mathbf{z},\mathbf{y}\rangle}e^{\ii \nu s}$,
for all $s,\nu\in\mathbb{R}$ and $\mathbf{z},\mathbf{y}\in\mathbb{R}^d$.
\subsubsection{\bf The Case $N=\mathbb{R}$}
Let $N:=\{(\mathbf{0},\mathbf{0},s):s\in\mathbb{R}\}$. Then $N$ is a central subgroup of $\mathbb{H}^d$ and  
$\theta_\mathbf{x}(\mathbf{0},\mathbf{0},s)=(\mathbf{0},\mathbf{0},s)$, for $\mathbf{x}\in H$ and $s\in\mathbb{R}$. 
Also, $\chi(\mathbb{R})=\{\mathbf{e}_\nu:\nu\in\mathbb{R}\}$, where $\mathbf{e}_\nu:=\mathbf{e}_{\mathbf{0},\nu}$. Let $\nu\in\mathbb{R}$ and $\xi:=\mathbf{e}_\nu\in\chi(N)$. Then, every $\psi\in \mathcal{C}_\xi(\mathbb{H}^d,\mathbb{R})$ satisfies 
\begin{equation}\label{cov.prop.H.R}
\psi(\mathbf{x},\mathbf{y},t)=e^{\ii\nu t}\psi(\mathbf{x},\mathbf{y},0)\ \ \ {\rm for}\ (\mathbf{x},\mathbf{y},t)\in\mathbb{H}^d.
\end{equation}
In this case, for $1\le p<\infty$, the norm $\|\psi\|_{(p)}$ is given by  
\[
\|\psi\|_{(p)}^p=\int_{\mathbb{R}^d}\int_{\mathbb{R}^d}|\psi(\mathbf{x},\mathbf{y},0)|^p\dd\mathbf{x}\dd\mathbf{y}.
\]
Let $(\mathbf{x},\mathbf{y},t)\in\mathbb{H}^d$, $f\in\mathcal{C}_c(\mathbb{H}^d)$ and $\psi\in \mathcal{C}_\xi(\mathbb{H}^d,\mathbb{R})$. Then 
\[
f\ast\psi(\mathbf{x},\mathbf{y},t)=\int_{\mathbb{R}^d}\int_{\mathbb{R}^d}\int_{-\infty}^\infty f(\mathbf{x}',\mathbf{y}',t')\psi(\mathbf{x}-\mathbf{x}',\mathbf{y}-\mathbf{y}',t-t'+\langle\mathbf{x}',\mathbf{y}'-\mathbf{y}\rangle)\dd\mathbf{x}'\dd\mathbf{y}'\dd t',
\]
and using (\ref{cov.prop.H.R}), we get 
\[
f\ast\psi(\mathbf{x},\mathbf{y},t)=e^{\ii\nu t}\int_{\mathbb{R}^d}\int_{\mathbb{R}^d}\int_{-\infty}^\infty f(\mathbf{x}',\mathbf{y}',t')\psi(\mathbf{x}-\mathbf{x}',\mathbf{y}-\mathbf{y}',0)e^{-\ii\nu(t'+\langle\mathbf{x}',\mathbf{y}-\mathbf{y}'\rangle)}\dd\mathbf{x}'\dd\mathbf{y}'\dd t'.
\]
\subsubsection{\bf The Case $N=\mathbb{R}^{d+1}$}
Let $N:=\{(\mathbf{y},s):s\in\mathbb{R},\mathbf{y}\in\mathbb{R}^d\}$.
Then, $N$ is a closed normal subgroup of $\mathbb{H}^d$. 
Suppose that $(\mathbf{z},\nu)\in\mathbb{R}^d\times\mathbb{R}$ and also $\xi:=\mathbf{e}_{\mathbf{z},\nu}$. So every $\psi\in \mathcal{C}_\xi(\mathbb{H}^d,\mathbb{R}^{d+1})$ satisfies 
\begin{equation}\label{cov.prop.H.Rd+1}
\psi(\mathbf{x},\mathbf{y},t)=e^{\ii\nu(t-\langle\mathbf{x},\mathbf{y}\rangle)}e^{\ii\langle\mathbf{z},\mathbf{y}\rangle}\psi(\mathbf{x},\mathbf{0},0)\ \ \ {\rm for}\ (\mathbf{x},\mathbf{y},t)\in\mathbb{H}^d.
\end{equation}
In this case, for $1\le p<\infty$, the norm $\|\psi\|_{(p)}$ is given by  
\[
\|\psi\|_{(p)}^p=\int_{\mathbb{R}^d}|\psi(\mathbf{x},\mathbf{0},0)|\dd\mathbf{x}.
\]
Let $(\mathbf{x},\mathbf{y},t)\in\mathbb{H}^d$, $f\in\mathcal{C}_c(\mathbb{H}^d)$ and $\psi\in \mathcal{C}_\xi(\mathbb{H}^d,\mathbb{R}^{d+1})$. Then  
\[
f\ast\psi(\mathbf{x},\mathbf{y},t)=\int_{\mathbb{R}^d}\int_{\mathbb{R}^d}\int_{-\infty}^\infty f(\mathbf{x}',\mathbf{y}',t')\psi(\mathbf{x}-\mathbf{x}',\mathbf{y}-\mathbf{y}',t-t'+\langle\mathbf{x}',\mathbf{y}'-\mathbf{y}\rangle)\dd\mathbf{x}'\dd\mathbf{y}'\dd t',
\]
and using (\ref{cov.prop.H.Rd+1}), we get 
\[
f\ast\psi(\mathbf{x},\mathbf{y},t)=e^{\ii\nu(t-\langle\mathbf{x},\mathbf{y}\rangle)}e^{\ii\langle\mathbf{z},\mathbf{y}\rangle}\int_{\mathbb{R}^d}\int_{\mathbb{R}^d}\int_{-\infty}^\infty f(\mathbf{x}',\mathbf{y}',t')\psi(\mathbf{x}-\mathbf{x}',\mathbf{0},0)e^{\ii\nu(\langle\mathbf{x},\mathbf{y}'\rangle-t')}e^{-\ii\langle\mathbf{z},\mathbf{y}'\rangle}\dd\mathbf{x}'\dd\mathbf{y}'\dd t'.
\]

%\newpage

{\bf Acknowledgement.}
This project has received funding from the European Union’s Horizon 2020 research and innovation programme under the Marie Sklodowska-Curie grant agreement No. 794305. The author gratefully acknowledges the supporting agency. The findings and opinions expressed here are only those of the author, and not of the funding agency.\\
The author would like to express his deepest gratitude to Vladimir V. Kisil for suggesting the problem that motivated the results in this article, and 
stimulating discussions.

\bibliographystyle{amsplain}

\begin{thebibliography}{10}

\bibitem{RB} R. Berndt, \textit{Representations of Linear Groups}. An Introduction Based on Examples from Physics and Number Theory. Springer-Vieweg, Wiesbaden, 2007.

\bibitem{Bour} N. Bourbaki, \textit{Integration. II. Chapters 7–9}, Elements of Mathematics(Berlin). Springer-Verlag, Berlin, 2004.

\bibitem{Brac} J. Braconnier, \textit{Sur les groupes topologiques localement compacts}, J. Math. Pures Appl. (9) 27 (1948), 1–85.

\bibitem{der00} A. Derighetti, \textit{Convolution Operators on Groups}, Lecture Notes of the Unione Matematica Italiana, 11. Springer, Heidelberg; UMI, Bologna, 2011. xii+171 pp.

\bibitem{der1983} A. Derighetti, \textit{On the multipliers of a quotient group}, Bull. Sci. Math. (2) 107 (1983), no. 1, 3-23.

\bibitem{Fei1} H.G. Feichtinger, \textit{Banach convolution algebras of functions II,} Monatsh. Math. 87 (1979), no. 3, 181-207.
%  
\bibitem{Fei2} H.G. Feichtinger, \textit{On a class of convolution algebras of functions,} Ann. Inst. Fourier 27 (1977)(3), vi, 135-162.
 
\bibitem{FollH} G.B. Folland, \textit{A Course in Abstract Harmonic Analysis}, Studies in Advanced Mathematics, CRC Press, Boca Raton, FL 1995.

\bibitem{FollP} G.B. Folland, \textit{Harmonic Analysis in Phase Space}, Annals of Mathematics Studies, 122. Princeton University Press, Princeton, NJ, 1989. x+277 pp.

\bibitem{AGHF.ch.normal.main}  A. Ghaani Farashahi, \textit{Harmonic analysis of covariant functions of characters of normal subgroups},  arXiv:2102.08901v2.

\bibitem{AGHF.cov.com} A. Ghaani Farashahi, \textit{Covariant functions of characters of compact subgroups}, arXiv:2102.07892.

\bibitem{AGHF.GWHG} A. Ghaani Farashahi, \textit{Generalized Weyl–Heisenberg (GWH) groups}, Analysis and Mathematical Physics. 4(3), pp. 187-197 (2014).

\bibitem{HR1} E. Hewitt and K.A. Ross, 
\textit{Abstract harmonic analysis. Vol. I: Structure of topological groups. Integration theory, group representations}, Die Grundlehren der mathematischen Wissenschaften, Bd. 115, Academic Press, Inc., Publishers, New York; Springer-Verlag, Berlin-G\"ottingen-Heidelberg, 1963,

\bibitem{kan.lau} E. Kaniuth, A.T.-M Lau, \textit{Fourier and Fourier-Stieltjes algebras on locally compact groups}, Mathematical Surveys and Monographs, 231. American Mathematical Society, Providence, RI, 2018.

\bibitem{kan.taylor} E. Kaniuth, K.F. Taylor, \textit{Induced Representations of Locally Compact Groups}, Cambridge Tracts in Mathematics, 197. Cambridge University Press, Cambridge, 2013. xiv+343 pp. 

\bibitem{kisil.BJMA} V. Kisil, \textit{Calculus of operators: covariant transform and relative convolutions}, Banach J. Math. Anal. 8(2014)2, 156-184.

\bibitem{kisil.book} V. Kisil, \textit{Geometry of M\"obius transformations. Elliptic, parabolic and hyperbolic actions of $SL_2(\mathbb{R})$}, 
Imperial College Press, London, 2012.

\bibitem{kisil1} V. Kisil, \textit{Relative convolutions. I. Properties and applications}, Adv. Math. 147 (1999), no. 1, 35-73.

\bibitem{MackII}  G.W. Mackey, \textit{Induced representations of locally compact groups. II. The Frobenius reciprocity theorem.}, Ann. of Math. (2) 58 (1953), 193-221.

\bibitem{MackI} G.W. Mackey, \textit{Induced representations of locally compact groups. I.}, Ann. of Math. (2) 55 (1952), 101-139.

\bibitem{pere} A. Perelomov, \textit{Generalized Coherent States and Their Applications}, Texts and Monographs in Physics. Springer-Verlag, Berlin, 1986.

\bibitem{50} H. Reiter, J.D. Stegeman,  \textit{Classical Harmonic Analysis}, 2nd Ed, Oxford University Press, New York, 2000.

\end{thebibliography}

\end{document}